\documentclass[a4paper,12pt]{article}
\author{}
\usepackage{amssymb}

\newtheorem{theorem}{Theorem}[section]
\newtheorem{corollary}{Corollary}[section]
\newtheorem{lemma}{Lemma}[section]
\newtheorem{proposition}{Proposition}[section]

\newtheorem{definition}{Definition}[section]
\newtheorem{conjecture}{Conjecture}[section]

\newcommand{\s}{\times}
\newcommand{\QQ}{{\bf Q}}
\newcommand{\CC}{{\bf C}}
\newcommand{\A}{{\bf A}}
\newcommand{\RR}{{\bf R}}
\newcommand{\FF}{{\bf F}}
\newcommand{\XX}{{\bf X}}

\newcommand{\MM}{\mathcal M}

\newcommand{\OO}{\mathcal O}

\newcommand{\mR}{{\mathcal R^\psi}}
\newcommand{\ts}{\widetilde\psi}
\newcommand{\mm}{\mathfrak m}
\newcommand{\TT}{{\bf T}}

\newcommand{\mM}{{\mathcal M^\psi}}
\newcommand{\orho}{\overline\rho}

\newcommand\qed{\hfill \rule{6pt}{6pt}}
\newenvironment{proof}{\noindent{\bf Proof\ }\hspace{1pt}}{\hspace{-5pt}\qed\par\bigskip}

\newcommand{\ZZ}{{\bf Z}}
\def\modulo{{\rm mod}}
\def\End{{\rm End}}
\def\Hom{{\rm Hom}}
\def\Frob{{\rm Frob}}
\def\det{{\rm det}}
\def\galois{{\rm Gal}}

\def\tr{{\rm tr}}
\def\ker{{\rm ker}}
\def\Im{{\rm Im}}
\def\dim{{\rm dim}}

\begin{document}

\title{Congruences between modular forms and related modules }
\author{Miriam Ciavarella}
\date{ }
\maketitle

\begin{abstract}
We fix $\ell$ a prime and let $M$ be an integer such that $\ell\not|M$; let $f\in S_2(\Gamma_1(M\ell^2))$ be a newform supercuspidal of fixed type related to the nebentypus, at $\ell$ and special at a finite set of primes. Let $\TT^\psi$ be the local quaternionic Hecke algebra associated to $f$. The algebra $\TT^\psi$ acts on a module $\mathcal M^\psi_f$ coming from the cohomology of a Shimura curve. Applying the Taylor-Wiles criterion and a recent Savitt's theorem, $\TT^\psi$ is the universal deformation ring of a global Galois deformation problem associated to $\orho_f$. Moreover $\mathcal M^\psi_f$ is free of rank 2 over $\TT^\psi$. If $f$ occurs at minimal level, by a generalization of a Conrad, Diamond and Taylor's result and by the classical Ihara's lemma, we prove a theorem of raising the level and a result about congruence ideals. The extension of this results to the non minimal case is an open problem.

\end{abstract}
\newpage

\section*{Introduction}

The principal aim of this article is to detect some results about isomorphism of complete intersection between an universal deformation ring and a local Hecke algebra and about cohomological modules free over an Hecke algebra. From this results, it is possible deduce that there is an isomorphism between the quaternionic cohomological congruence module for a modular form and the classical congruence module.\\
Our work take place in a context of search which has its origin in the works of Wiles and Taylor-Wiles on the Shimura-Taniyama-Weil conjecture. Recall that the problem addressed by Wiles in \cite{wi} is to prove that a certain ring homomorphism $\phi:\mathcal R_\mathcal D\to{\bf T}_\mathcal D$ is an isomorphism (of complete intersection), where $\mathcal R_\mathcal D$ is the universal deformation ring for a mod $\ell$ Galois representation arising from a modular form and ${\bf T}_\mathcal D$ is a certain Hecke algebra.  \\
Extending our results to a more general class of representations letting free the ramification to a finite set of prime we need a small generalization of a result of Conrad Diamond and Taylor. We will describe it, using a recent Savitt's theorem and we deduce from it two interesting results about congruences.\\ 
Our first result extends a work of Terracini \cite{Lea} to a more general class of types and allows us to work with modular forms having a non trivial nebentypus.
Our arguments are largely identical to Terracini's in many places; the debt to Terracini's work will be clear throughout the paper. One important change is that since we will work with Galois representations which are not semistable at $\ell$ but only potentially semistable,
we use a recent Savitt's theorem \cite{savitt},  that prove a conjecture of Conrad, Diamond and Taylor (\cite{CDT}, conjecture 1.2.2 and conjecture 1.2.3),  on the size of certain deformation rings parametrizing potentially Barsotti-Tate Galois representations, extending results of Breuil and M\'{e}zard (conjecture 2.3.1.1 of \cite{BM}) (classifying Galois lattices in semistable representations in terms of \lq\lq strongly divisible modules \rq\rq) to the potentially crystalline case in Hodge-Tate weights $(0,1)$. 

Given a prime $\ell$, we fix a newform $f\in S_2(\Gamma_0(N\Delta'\ell^2),\psi)$ with nebentypus $\psi$ of order prime to $\ell$, special at primes dividing $\Delta'$ and such that its local representation $\pi_{f,\ell}$ of $GL_2(\QQ_\ell)$, is associated to a fixed regular character $\chi$ of $\FF_{\ell^2}^\s$ satisfying $\chi|_{\ZZ_\ell^\s}=\psi_\ell|_{\ZZ_\ell^\s}$. We consider the residual Galois representation $\orho$ associated to $f$. We denote by $\TT^\psi$ the $\ZZ_\ell$-subalgebra of $\prod_{h\in\mathcal B}\OO_{h,\lambda}$ where $\mathcal B$ is the set of normalized newforms in $S_2(\Gamma_0(N\Delta'\ell^2),\psi)$ which are supercuspidal of type $\tau=\chi^\sigma\oplus\chi$  at $\ell$ and whose associated representation is a deformation of $\orho$. By the Jacquet-Langlands correspondence  and the Matsushima-Murakami-Shimura isomorphism, one can see such forms in a local component $\mM$ of the $\ell$-adic cohomology of a  Shimura curve. By imposing suitable conditions on the type $\tau$, we describe for each prime $p$ dividing the level, a local deformation condition of $\orho_p$ and applying the Taylor-Wiles criterion in the version of Diamond \cite{D} and Fujiwara, we prove that the algebra $\TT^\psi$ is characterized as the universal deformation ring $\mR$ of our global Galois deformation problem. 
We point out that in order to prove the existence of a family of sets realizing simultaneously the conditions of a Taylor-Wiles system, we make large use of Savitt's theorem \cite{savitt}: assuming the existence of a newform $f$ as above, the tangent space of the deformation functor has dimension one over the residue field. Our first resul is the following:
\begin{theorem}
\begin{itemize}
\item[a)] $\Phi:\mR\to\TT^\psi$ is an isomorphism of complete intersection;
\item[b)] $\mM$ is a free $\TT^\psi$-module of rank 2.
\end{itemize}
\end{theorem}

\noindent We observe that in \cite{CDT} the authors assume that the type $\tau$ is strongly acceptable for $\orho_\ell$.  In this way they assure the existence of a modular form under their hyphotesis. Since we are interested to study the quaternionic cohomolgical module associated to the modular form $f$, as a general hyphotesis we suppose that there exist  a modular form $f$ in our conditions, in other words we are assuming that uor cohomological module is not empty. 
\noindent Keeping the definitions as in \cite{CDT}, Savitt's result allows to suppress the assumption of acceptability in the definition of strongly acceptability, thus it is possible to extend Conrad, Diamond and Taylor's result \cite{CDT} relaxing the hypotheses on the residual representation.\\
\noindent Under the hypothesis that $f$ occurs with minimal level (i.e. the ramification at primes $p$ dividing the Artin conductor of the Galois representation $\rho_f$ is equal to the ramification of $\orho_f$ at $p$) 
the module $\mM$, used to construct the Taylor-Wiles system, can be also seen as a part of a module $\mathcal M^{\rm mod}$ coming from the cohomology of a modular curve, as described in \cite{CDT} \S5.3. Applying the extended Conrad, Diamond and Taylor's methods and by the Ihara's lemma for the cohomology of modular curves \cite{CDT}, the first part of our result can be extended by allowing the ramification on a set of primes $S$ disjoint from $N\Delta'\ell$. In this way it is possible to obtain results of the form:
\begin{itemize}
\item $\Phi_S:\mathcal R^\psi_S\to\TT^\psi_S$ is an isomorphism of complete intersections,
\end{itemize}
where $\mathcal R^\psi_S$ is an universal deformation ring letting free the ramification at primes in $S$, $\TT^\psi_S$ is a local Hecke algebra. 
We observe that, since there is not an analogous of the Ihara's lemma for the cohomology of the Shimura curve, we don't have any information about the correspondent module $\mathcal M^\psi_S$ coming from the cohomology of a Shimura curve. In particular, in the general case $S\not=\emptyset$,  it is not possible to show that  $\mathcal M^\psi_S$ is free over $\TT^\psi_S$  

We observe that as a consequence of the generalization of the Cornrad, Diamond and Taylor's result, two results about raising the level of modular forms and about the ideals of congruence follows. 
Let $S_1,S_2$ be two subsets of $\Delta_2$, we slightly modify the deformation problem, by imposing the condition sp at primes $p$ in $S_2$ and by allowing ramification at primes in $S_1$. Let we denote by $\eta_{S_1,S_2}$ the congruence ideal of a modular form relatively to the set $\mathcal B_{S_1,S_2}$ of the newforms  of weight 2, Nebentypus $\psi$, level dividing $N\Delta_1\Delta_2\ell$ which are supercuspidal of type $\tau$ at $\ell$, special at primes in $\Delta_1S_2$. We prove that there is an isomorphism of complete intersections between the universal deformation ring $\mathcal R^\psi_{S_1,S_2}$ and the Hecke algebra $\TT^\psi_{S_1,S_2}$ acting on the space $\mathcal B_{S_1,S_2}$ and $$\eta_{\Delta_2,\emptyset}=C\eta_{S_1,S_2}$$ where $C$ is a constant depending of the modular form. In particular  we prove the following result:
\begin{theorem}
Let $f=\sum a_nq^n$ be a normalized newform in $S_2(\Gamma_0(M\ell^2),\psi)$ supercuspidal of type $\tau=\chi\oplus\chi^\sigma$ at $\ell$, special at primes in a finite set $\Delta'$, there exist $g\in S_2(\Gamma_0(qM\ell^2),\psi)$ supercuspidal of type $\tau$ at $\ell$, special at every prime $p|\Delta'$ such that $f\equiv g\ \modulo\ \lambda$ if and only if $$a_q^2\equiv\psi(q)(1+q)^2\ \modulo\ \lambda$$ where $q$ is a prime such that $(q,M\ell^2)=1,$ $q\not\equiv-1\ \modulo\ \ell$. 
\end{theorem}

We observe that our results concerning the cohomological modules, holds only at the minimal level since a quaternionic analogue of the Ihara's lemma is not available in this case. Let $S$ be a finite set of primes not dividing $M\ell$; we fix $f\in S_2(\Gamma_0(N\Delta'\ell^2 S),\psi)$ supercuspidal of type $\tau$ at $\ell$, special at primes $p|\Delta'$. If we modify our Galois deformation problem allowing ramification at primes in $S$, we obtain a new universal deformation ring $\mathcal R_S$ and a new Hecke algebra $\TT^\psi_S$ acting on the newforms giving rise to such representation.  We make the following conjecture:
\begin{conjecture}\label{conj1}
\begin{itemize}
\item $\mathcal R_S\to\TT^\psi_S$ is an isomorphism of complete intersection;
\item let $\mathcal M^\psi_S$ be the module $H^1(\XX_1(NS),\OO)_{\mm_S}^{\widehat\psi}$ coming from the cohomology of the Shimura curve $\XX_1(NS)$ associated to the open compact subgroup of $B_\A^{\s,\infty}$, $V_1(NS)=\prod_{p\not|NS\ell}R_p^\s\prod_{p|NS}K_p^1(N)\s(1+u_\ell R_\ell)$ where  $K_p^1(N)$ is defined in section \ref{shi}, and $u_\ell$ is a uniformizer of $B_\ell^\s$. $\mathcal M^\psi_S$ is a free $\TT^\psi_S$-module of rank 2. 
\end{itemize}
\end{conjecture}
\noindent Conjecture \ref{conj1} easily follows from the following conjecture:

\begin{conjecture}\label{noscon}
Let $q$ be a prime number such that $q\not|N\Delta'\ell^2$. We fix a maximal non Eisenstein ideal of the Hecke algebra  $\TT_0^{\widehat\psi}(N)$ acting on the group $H^1(\XX_1(N),\OO)^{\widehat\psi}$. Let $\XX_1(N)$ be the Shimura curve  $$\XX_1(N)=B^\s\setminus B_\A^\s/K_\infty^+V_1(N)$$ where $$V_1(N)=\prod_{p\not|N\ell}R_p^\s\prod_{p|N}K_p^1(N)\s(1+u_\ell R_\ell)$$ where $K_p^1(N)$ is defined in section \ref{shi}, and $u_\ell$ is a uniformizer of $B_\ell^\s$. The map $$\alpha_\mm:H^1(\XX_1(N),\OO)_\mm^{\widehat\psi}\s H^1(\XX_1(N),\OO)_\mm^{\widehat\psi}\to H^1(\XX_1(Nq),\OO)_{\mm^q}^{\widehat\psi}$$ is such that   $\alpha\otimes_\OO k$ is injective, where $\mm^q$ is the inverse image of the ideal $\mm$ under the natural map $\TT_0^{\widehat\psi}(Nq)\to\TT_0^{\widehat\psi}(N)$ and $k=\OO/\lambda$.

\end{conjecture}

\noindent This conjecture would provide an analogue for the Shimura curves  of the Ihara's Lemma in case $\ell|\Delta$. In \cite{DT} and in \cite{DTi}, Diamond and Taylor show that if $\ell$ not divides the discriminant of the indefinite quaternion algebra, then the analogue of conjecture \ref{noscon} holds.

\section{Notations}
 
\noindent For a rational prime $p$, $\ZZ_p$ and $\QQ_p$ denote the ring of $p$-adic integers and the field of $p$-adic numbers, respectively. If $A$ is a ring, then $A^\s$ denotes the group of invertible elements of $A$. We will denote by $\A$ the ring of rational ad\'eles, and by $\A^\infty$ the finite ad\'eles.\\
Let  $B$ be a quaternion algebra on $\QQ$, we will denote by $B_\A$ the adelization of $B$, by $B_\A^\times$ the topological group of invertible elements in $B_\A$ and $B_\A^{\times,\infty}$ the subgroup of finite ad\'eles.\\
Let $R$ be a maximal order in $B$.For a rational place $v$ of $\QQ$ we put $B_v=B\otimes_\QQ\QQ_v$; if $p$ is a finite place we put $R_p=R\otimes_\ZZ\ZZ_p$. \\
 If $p$ is a prime not dividing the discriminant of $B$, included $p=\infty$, we fix an isomorphism $i_p:B_p\to M_2(\QQ_p)$ such that if $p\not=\infty$ we have $i_p(R_p)=M_2(\ZZ_p)$.\\
We write $GL_2^+(\RR)=\{g\in GL_2(\RR)|\ det\ g>0\}$ and $K_\infty=\RR^\times O_2(\RR),$ $K_\infty^+=\RR^\times SO_2(\RR).$
If $K$ is a field, let $\overline K$ denote an algebraic closure of $K$; we put $G_K=\galois(\overline K/K)$. For a local field $K$, $K^{unr}$ denotes the maximal unramified extension of $K$ in $\overline K$; we put $I_K=\galois(\overline K/K^{unr})$, the inertia subgroup of $G_K$. For a prime $p$ we put $G_p=G_{\QQ_p}$, $I_p=I_{\QQ_p}$. If $\rho$ is a representation of $G_\QQ$, we write $\rho_p$ for the restriction of $\rho$ to a decomposition group at $p$.

\section{The local Hecke algebra $\TT^\psi$}\label{de}

We fix a prime $\ell>2$. Let $\ZZ_{\ell^2}$ denote the integer ring of $\QQ_{\ell^2}$, the unramified quadratic extension of $\QQ_\ell$. Let $M\not=1$ be a square-free integer not divisible by $\ell$. 
We fix $f$ an eigenform in $S_2(\Gamma_1(M\ell^2))$, then $f\in S_2(\Gamma_0(M\ell^2),\psi)$ for some Dirichlet character $\psi:(\ZZ/M\ell^2\ZZ)^\s\to\overline\QQ^\s.$\\
For abuse of notation, let $\psi$ be the adelisation of the Dirichlet character $\psi$ and we denote by $\psi_p$ the composition of $\psi$ with the inclusion $\QQ_p^\s\to \A^\s$.\\
We fix a regular character  $\chi:\ZZ_{\ell^2}^\s\to \overline\QQ^\s$ of conductor $\ell$ such that $\chi|_{\ZZ_\ell^\s}=\psi_\ell|_{\ZZ_\ell^\s}$
and we extend $\chi$ to $\QQ_{\ell^2}^\s$ by putting $\chi(\ell)=-\psi_\ell(\ell)$. We observe that $\chi$ is not uniquely determined by $\psi$ and, if we fix an embedding of $\overline\QQ$ in $\overline\QQ_\ell$, we can ragard the values of $\chi$ in this field.\\
Since, by local classfield theory, $\chi$ can be regarded as a character of $I_\ell$ and 
we can consider the type $\tau=\chi\oplus\chi^\sigma:I_\ell\to GL_2(\overline\QQ_\ell)$, where $\sigma$ denotes the complex conjugation.\\ 
We fix a decomposition $M=N\Delta'$ where $\Delta'$ is a product of an odd number of primes. If we shoose $f\in S_2(\Gamma_1(M\ell^2))$ such that the automorphic representation $\pi_f=\otimes_v\pi_{f,v}$ of $GL_2(\A)$ associated to $f$ is supercuspidal of type $\tau=\chi\oplus\chi^\sigma$ at $\ell$  and special at every primes $p|\Delta'$, then $\pi_{f,\ell}=\pi_\ell(\chi)$, where $\pi_\ell(\chi)$ is the representation of $GL_2(\QQ_\ell)$ associated to $\chi$,  with  central character  $\psi_\ell$ and conductor $\ell^2$ (see \cite{Ge}, \S 2.8). Moreover, under our hypotheses, the nebentypus $\psi$ factors through $(\ZZ/N\ell\ZZ)^\s$. As a general hypothesis, we assume that  $\psi$ has order prime to $\ell$.\\ 
Let $WD(\pi_\ell(\chi))$ be the $2$-dimensional representation of the Weil-Deligne group at $\ell$ associated to $\pi_\ell(\chi)$ by local Langlands correspondence. Since by  local  classfield theory, we can identify $\QQ_{\ell^2}^\s$ with $W_{\QQ_{\ell^2}}^{ab}$, we can see $\chi$ as a character of $W_{\QQ_{\ell^2}}$;   by (\cite{Ca} \S 11.3), we have 
\begin{equation}\label{ind}
WD(\pi_\ell(\chi))=Ind^{W_{\QQ_{\ell}}}_{W_{\QQ_{\ell^2}}}(\chi)\otimes|\ |_\ell^{-1/2}.
\end{equation}

\noindent Let $\rho_f:G_\QQ\to GL_2(\overline\QQ_\ell)$ be the Galois representation associated to $f$ and $\orho:G_\QQ\to GL_2(\overline\FF_\ell)$ be its reduction modulo $\ell$.\\
As in \cite{Lea}, we impose the following conditions on $\orho$:
\begin{equation}\label{con1}
\orho\ {\rm is\ absolutely\ irreducible};
\end{equation} 
\begin{equation}\label{cond2} 
{\rm if}\ p|N\ {\rm then}\ \orho(I_p)\not=1;
\end{equation}
\begin{equation} \label{rara}
{\rm if}\ p|\Delta'\ {\rm and}\ p^2\equiv 1 \modulo\ \ell\ {\rm then}\ \orho(I_p)\not=1;
\end{equation}
\begin{equation}\label{end}
\End_{\overline\FF_\ell[G_\ell]}(\orho_\ell)=\overline\FF_\ell.
\end{equation}
\begin{equation}\label{c3}
{\rm if}\ \ell=3,\ \orho\ \ {\rm is\ not\ induced\ from\ a\ character\ of\ }\ \QQ(\sqrt{-3}).
\end{equation}
\noindent Let $K=K(f)$ be a finite extension of $\QQ_\ell$ containing $\QQ_{\ell^2}$, $\Im(\psi)$ and the eigenvalues for $f$ of all Hecke operators. Let $\OO$ be the ring of integers of $K$, $\lambda$ be a uniformizer of $\OO$, $k=\OO/(\lambda)$ be the residue field.

\noindent Let $\mathcal B$ denote the set of normalized newforms $h$ in $S_2(\Gamma_0(M\ell^2),\psi)$ which are supercuspidal of type $\chi$ at $\ell$, special at primes dividing $\Delta'$ and whose associated representation $\rho_h$ is a deformation of $\orho$. For $h\in\mathcal B$, let $h=\sum_{n=1}^\infty a_n(h)q^n$ be the $q$-expansion of $h$ and let $\OO_h$ be the $\OO$-algebra generated in $\QQ_\ell$ by the Fourier coefficients of $h$. Let $\TT^\psi$ denote the sub-$\OO$-algebra of $\prod_{h\in\mathcal B}\OO_h$ generated by the elements $\widetilde T_p=(a_p(h))_{h\in\mathcal B}$ for $p\not|M\ell$.\\

\section{Deformation problem}

Our next goal is to state a global Galois deformation condition of $\orho$ which is a good candidate for having $\TT^\psi$ as an universal deformation ring.

\subsection{The global deformation condition of type $(\rm{sp},\tau,\psi)_Q$}\label{univ}

First of all we observe that our local Galois representation $\rho_{f,\ell}=\rho_\ell$ is of type $\tau$ (\cite{CDT}).\\
We let $\Delta_1$ be the product of primes $p|\Delta'$ such that $\orho(I_p)\not=1$, and $\Delta_2$ be the product of primes $p|\Delta'$ such that $\orho(I_p)=1$.\\
We denote by $\mathcal C_\OO$ the category of local complete noetherian $\OO$-algebras with residue field $k$. Let $\epsilon:G_p:\to\ZZ_\ell^\s$ be the cyclotomic character and $\omega:G_p\to\FF_\ell^\s$ be its reduction mod $\ell$. 
By analogy with \cite{Lea}, We define the global deformation condition of type $(\rm{sp},\tau,\psi)_Q$:
\begin{definition}\label{def}
Let $Q$ be a square-free integer, prime to $M\ell$. We consider the functor $\mathcal F_Q$ from $\mathcal C_\OO$ to the category of sets which associate to an object $A\in\mathcal C_\OO$ the set of strict equivalence classes of continuous homomorphisms $\rho:{G_\QQ}\to GL_2(A)$ lifting $\orho$ and satisfying the following conditions:
\begin{itemize}
\item[a$_Q$)] $\rho$ is unramified outside $MQ\ell$;
\item[b)] if $p|\Delta_1N$ then $\rho(I_p)\simeq\orho(I_p)$ ;
\item[c)] if $p|\Delta_2$ then $\rho_p$ satisfies the sp-condition, that is $\tr(\rho(F))^2=(p\mu(p)+\mu(p))^2=\psi_p(p)(p+1)^2$ for a lift $F$ of $\Frob_p$ in $G_p$;
\item[d)] $\rho_\ell$ is weakly of type $\tau$;
\item[e)] $\det(\rho)=\epsilon\psi$, where $\epsilon:G_\QQ\to\ZZ_\ell^\s$ is the cyclotomic character. 
\end{itemize} 
\end{definition}

\noindent It is easy to prove that the functor $\mathcal F_Q$ is representable.

\noindent Let $\mathcal R^\psi_Q$ be the universal ring associated to the functor $\mathcal F_Q$. We put $\mathcal F=\mathcal F_0$, $\mR=\mathcal R^\psi_0$.\\
We observe that if $\orho(I_p)=1$, by the Ramanujan-Petersson conjecture proved by Deligne, the sp-condition rules out thouse defomations of $ \orho$ arising from  modular forms which are not special at $p$.   This space includes the restrictions to $G_p$ of representations coming from forms in $S_2(\Gamma_0(N\Delta'\ell^2),\psi)$ which are special at $p$, but it does not contain those coming from principal forms in $S_2(\Gamma_0(N\Delta'\ell^2),\psi)$.

\section{Cohomological modules coming from the Shimura curves}\label{shi}

We fix a prime $\ell>2$. Let $\Delta'$ be a product of an odd number of primes, different from $\ell$. We put $\Delta=\ell\Delta'$. Let $B$ be the indefinite quaternion algebra over $\QQ$ of discriminant $\Delta$. Let $R$ be a maximal order in $B$. Let $N$ be an integer prime to $\Delta$. We put $$K_p^0(N)=i_p^{-1}\left\{\left(
\begin{array}
[c]{cc}%
a & b\\
c & d
\end{array}
\right)\in GL_2(\ZZ_p)\ |\ c\equiv 0\ \modulo\ N\right\}$$

$$K_p^1(N)=i_p^{-1}\left\{\left(
\begin{array}
[c]{cc}%
a & b\\
c & d
\end{array}
\right)\in \ GL_2(\ZZ_p)\ |\ c\equiv 0\ \modulo\ N,\ a\equiv 1\ \modulo\ N\right\}.$$
Let $s$ be a prime $s\not|N\Delta$. We define $$V_0(N,s)=\prod_{p\not|Ns}R^\s_p\s \prod_{p|N}K_p^0(N)\s K_s^1(s^2)$$ and $$V_1(N,s)=\prod_{p\not|N\ell s}R^\s_p\s\prod_{p|N}K_p^1(N)\s K_s^1(s^2)\s (1+u_\ell R_\ell).$$ 
We observe that there is an isomorphism $$V_0(N,s)/V_1(N,s)\simeq(\ZZ/N\ZZ)^\s\s\FF_{\ell^2}^\s.$$ We will consider  the character $\widehat\psi$ of $V_0(N,s)$ with kernel $V_1(N,s)$ defined as follow: $$\widehat\psi:=\prod_{p|N}\psi_p\s\chi:(\ZZ/N\ZZ)^\s\s\FF_{\ell^2}^\s\to\CC^\s$$
and we shall consider the space $S_2(V_0(N,s),\widehat\psi)$ of quaternionic modular forms with nebentypus $\widehat\psi$.\\ 
For $i=0,1$ let $\Phi_i(N,s)=(GL_2^+(\RR)\s V_i(N,s))\cap B_\QQ^\s$  
and we consider the Shimura curves: $$\XX_i(N,s)=B_\QQ^\s\setminus B^\s_\A/K_\infty^+\s V_i(N,s).$$  
The finite commutative group $\Omega=(\ZZ/N\ZZ)^\s\s\FF_{\ell^2}^\s$  naturally acts on the $\OO$-module $H^*(\XX_1(N,s),\OO)$ via its action on $\XX_1(N,s)$. Since there is an injection of $H^*(\XX_1(N,s),\OO)$ in $H^*(\XX_1(N,s),K)$, by (\cite{H2}, \S 7) the cohomology group $H^1(\XX_1(N,s),\OO)$ is also equipped with the action of Hecke operator $T_p$, for $p\not=\ell$ and diamond operators $\langle n\rangle$ for $n\in(\ZZ/N\ZZ)^\s$.  The Hecke action commutes with the action of $\Omega$, since we do not have a $T_\ell$ operator. The two actions are $\OO$-linear.\\
We can write $\Omega=\Omega_1\s\Omega_2$ where $\Omega_1$ is the $\ell$-Sylow subgroup of $\Omega$ and $\Omega_2$ is the subgroup of $\Omega$ with order prime to $\ell$. Since $\Omega_2\subseteq\Omega$, $\Omega_2$ acts on  $H^*(\XX_1(N,s),\OO)$ and so $H^*(\XX_1(N,s),\OO)=\bigoplus_\varphi H^*(\XX_1(N,s),\OO)^\varphi$ where $\varphi$ runs over the characters of $\Omega_2$ and $H^*(\XX_1(N,s),\OO)^\varphi$ is the sub-Hecke-module of $H^*(\XX_1(N,s),\OO)$ on which $\Omega_2$ acts by the character $\varphi$. Since, by hypothesis, $\psi$ has order prime to $\ell$, $H^*(\XX_1(N,s),\OO)^{\widehat\psi}=H^*(\XX_1(N,s),\OO)^\varphi$ for some character $\varphi$ of $\Omega_2$. So $H^*(\XX_1(N,s),\OO)^{\widehat\psi}$ is a direct summand of $H^*(\XX_1(N,s),\OO)$.
{}
\noindent It follows easily from the Hochschild-Serre spectral sequence that $$H^*(\XX_1(N,s),\OO)^{\widehat\psi}\simeq H^*(\XX_0(N,s),\OO({\widehat\psi}))$$ where $\OO({\widehat\psi})$ is the sheaf $B^\s\setminus B_\A^\s\s\OO/K_\infty^+\s V_0(N,s),$ $B^\s$ acts on $B^\s_\A\s\OO$ on the left by $\alpha\cdot(g,m)=(\alpha g,m)$ and $K_\infty^+\s V_0(N,s)$ acts on the right by $(g,m)\cdot v=(g,m)\cdot(v_\infty,v^\infty)=(gv,{\widehat\psi}(v^\infty)m)$ where $v_\infty$ and $v^\infty$ are respectively the infinite and finite part of $v$ . By translating to the cohomology of groups we obtain (see \cite{sl}, Appendix) $H^1(\XX_1(N,s),\OO)^{\widehat\psi}\simeq H^1(\Phi_0(N,s),\OO(\ts)),$ where $\ts$ is the restriction of ${\widehat\psi}$ to $\Phi_0(N,s)/\Phi_1(N,s)$ and $\OO(\widetilde\psi)$ is $\OO$ with the action of $\Phi_0(N,s)$ given by $a\mapsto\widetilde\psi^{-1}(\gamma)a$.

\noindent It is well know the Hecke action on $H^1(\Phi_0(N,s),\OO(\ts))$ and the structure of the module $H^1(\XX_1(N,s),K)^{\widehat\psi}$ over the Hecke algebra. Let $\TT^{\widehat\psi}_0(N,s)$ be the $\OO$-algebra generated by the Hecke operators $T_p$, $p\not=\ell$  and the diamond operators, acting on $H^1(\XX_1(N,s),\OO)^{\widehat\psi}$.

\begin{proposition}\label{ra}
$H^1(\XX_1(N,s),K)^{\widehat\psi}$ is free of rank 2 over $$\TT^{\widehat\psi}_0(N,s)\otimes K=\TT^{\widehat\psi}_0(N,s)_K.$$
\end{proposition}

\noindent The proof of proposition \ref{ra}, easy follows from the following lemmas:
\begin{lemma}
Let  $L\supseteq K$ be two Galois fields, let $V$ be a vector space on $K$ and let $T$ be a $K$-algebra. If $V\otimes L$ is free of rank $n$ on $T\otimes L$ then $V$ is free of rank $n$ on $T$. 
\end{lemma}

\begin{proof}
Let $G=\galois(L/K)$. By the Galois theory, since $V\otimes L$ is free of rank $n$ on $T\otimes L$, we have
$V\simeq(V\otimes L)^G\simeq((T\otimes L)^n)^G\simeq((T\otimes L)^G)^n\simeq T^n.$ \end{proof}

\begin{lemma}
Let $\TT_0^{\widehat\psi}(N,s)_\CC$ denote the algebra generated over $\CC$ by the operators $T_p$  for $p\not=\ell$  
acting on $H^1(\XX_1(N,s),\CC)^{\widehat\psi}$. Then $H^1(\XX_1(N,s),\CC)^{\widehat\psi}$ is free of rank 2 over $\TT^{\widehat\psi}_0(N,s)_\CC$.
\end{lemma} 

\noindent We observe that the proof of this lemma follows by the same analysis explained in (\cite{Lea}, Proof of proposition 1.2), defining an  homomorphism $$JL:S_2(V_1(N,s))\to S_2(\Gamma_0(s^2\Delta')\cap\Gamma_1(N\ell^2))$$  which is injective when restricted to the space $S_2(V_0(N,s),\widehat\psi)$ and equivariant by Hecke operators.

\section{The $\OO$-module $\mM$ }\label{tw}

Throughout this section, we largely mirror section 3 of \cite{Lea}, and we formulate a result that generalizes of a result of Terracini to the case of nontrivial nebentypus.\\
We set $\Delta=\Delta'\ell$; let $B$ be the indefinite quaternion algebra over $\QQ$ of discriminant $\Delta$. Let $R$ be a maximal order in $B$.\\
It is convenient to choose an auxiliary prime $s\not|M\ell$, $s>3$ such that no lift of $\orho$ can be  ramified at $s$; such prime exists by \cite{DT}, Lemma 2. We consider the group $\Phi_0=\Phi_0(N,s)$; it easy to verify that the group $\Phi_0$ has not elliptic elements (\cite{Lea}).

\noindent There exists an eigenform $\widetilde f\in S_2(\Gamma_0(Ms^2\ell^2),\psi)$ such that $\rho_f=\rho_{\widetilde f}$ and $T_s\widetilde f=0$. By the Jacquet-Langlands correspondence, the form $\widetilde f$ determines a character $\TT^{\widehat\psi}_0(N,s)\to k$ 
sending the operator $t$ in the class $\modulo\ \lambda$ of the eigenvalue of $t$ for $\widetilde f$. The kernel of this character is a maximal ideal  $\mm$ in $\TT^{\widehat\psi}_0(N,s)$. We define $\mM=H^1(\XX_1(N,s),\OO)^{\widehat\psi}_\mm.$
By combining Proposition 4.7 of \cite{DDT} with the Jacquet-Langlands correspondence we see that there is a natural isomorphism $\TT^\psi\simeq\TT_0^{\widehat\psi}(N,s)_\mm.$ Therefore, by Proposition \ref{ra}, $\mM\otimes_\OO K\  {\rm is\ free\ of\ rank\ 2\ over}\ \TT^\psi\otimes_\OO K.$

\noindent Let $\mathcal B$ denote the set of newforms $h$  of weight two, nebentypus $\psi$, level dividing $M\ell^2s^2$, special at primes $p$ dividing $\Delta'$, supercuspidal of type $\tau$ at $\ell$ and such that $\orho_h\sim\orho$. For a newforms $h\in\mathcal B$, we let $K_h$ denote the field over $\QQ_\ell$ generated by its coefficients $a_n(h)$, $\OO_h$ denote the ring of integers of $K_h$ and let $\lambda$ be a uniformizer of $\OO_h$. We let $A_h$ denote the subring of $\OO_{h}$ consisting of those elements whose reduction mod $\lambda$ is in $k$. We know that with respect to some basis, we have 
$\rho_h:G_\QQ\to GL_2(A_h)$ a deformation of $\orho$ satisfying our global deformation problem.\\ 
The universal property of $\mR$ furnishes a unique homomorphism $\pi_h:\mR\to A_h$  such that the composite $G_\QQ\to GL_2(\mR)\to GL_2(A_h)$ is equivalent to $\rho_h$. Since $\mR$ is topologically generated by the traces of $\rho^{\rm univ}(\Frob_p)$ for $p\not=\ell$, (see \cite{Ma}, \S1.8), we conclude that  the map 
$\mR\to\prod_{h\in\mathcal B}\OO_h$ such that $r\mapsto(\pi_h(r))_{h\in\mathcal B}$
has image $\TT^\psi$. Thus there is a surjective homomorphism of $\OO$-algebras $\Phi:\mR\to\TT^\psi.$ Our goal is to prove the following 

\begin{theorem}\label{goal}
\begin{itemize}
\item[a)] $\mR$ is complete intersection of dimension 1;
\item[b)] $\Phi:\mR\to\TT^\psi$ is an isomorphism;
\item[c)] $\mM$ is a free $\TT^\psi$-module of rank 2.
\end{itemize}
\end{theorem}

\subsection{Proof of theorem \ref{goal}}

\noindent In order to prove theorem \ref{goal}, we shall apply the Taylor-Wiles criterion in the version of Diamond and Fujiwara and we continue to follow section 3 of \cite{Lea} closely.\\
\noindent We shall prove the existence of a family $\mathcal Q$ of finite sets $Q$ of prime numbers, not dividing $M\ell$ and of a $\mR_Q$-module $\mM_Q$ for each $Q\in\mathcal Q$ such  that the system $(\mR_Q,\mM_Q)_{Q\in\mathcal Q}$ satisfies the  conditions (TWS1), (TWS2), (TWS3), (TWS4), (TWS5) and (TWS6) of \cite{Lea}.

\noindent If this conditions are satisfy for the family $(\mR_Q,\mM_Q)_{Q\in\mathcal Q}$, it will be called a {\bf Taylor-Wiles system}  for $(\mR,\mM)$. Then theorem \ref{goal} will follow from the isomorphism criterion (\cite{D}, theorem 2.1) developed by Wiles, Taylor-Wiles.\\
As in section 3.1 of \cite{Lea}, let $Q$ be a finite set of prime numbers not dividing $N\Delta$ and such that
\begin{itemize}
\item[(A)] $q\equiv 1\ \modulo\ \ell,\ \ \forall q\in Q$;
\item[(B)] if $q\in Q$, $\orho(\Frob_q)$ has distinct eigenvalues $\alpha_{1,q}$ and $\alpha_{2,q}$ contained in $k$.  
\end{itemize}
We will define  the modules $\MM_Q$.
If $q\in Q$ we put $$K'_q=\left\{\alpha\in R_q^\s\ |\ i_q(\alpha)\in \left(
\begin{array}
[c]{cc}%
H_q & *\\
q\ZZ_q & *
\end{array}
\right)\right\}$$ where $H_q$ is the subgroup of $(\ZZ/q\ZZ)^\s$ consisting of elements of order prime to $\ell$. By analogy with the definition of $V_Q$ in section 3.3 of \cite{Lea}, we define $$V_Q'(N,s)=\prod_{p\not|NQs}R_p^\s\s\prod_{p|N}K_p^0(N)\s K_s^1(s^2)\s\prod_{q|Q}K'_q$$
$$V_Q(N,s)=\prod_{p\not|NQs}R_p^\s\s\prod_{p|NQ}K_p^0(NQ)\s K_s^1(s^2)$$
$$\Phi_Q=(GL_2(\RR^+)\s V_Q(N,s))\cap B^\s,\ \ \Phi_Q'=(GL_2(\RR^+)\s V_Q'(N,s))\cap B^\s.$$

\noindent Then $\Phi_Q/\Phi_Q'\simeq\Delta_Q$ acts on $H^1(\Phi_Q',\OO(\widetilde\psi))$. Let $\TT_Q^{'\widehat\psi}(N,s)$ (resp. $\TT_Q^{\widehat\psi}(N,s)$) be the Hecke $\OO$-algebra generated by the Hecke operators $T_p$, $p\not=\ell$ and the diamond operators ( that are those Hecke operators coming from $\Delta_Q$) , acting on $H^1(\XX'_Q(N,s),\OO)^{\widehat\psi}$  (resp. $H^1(\XX_Q(N,s),\OO)^{\widehat\psi}$) where $\XX'_Q(N,s)$ (resp. $\XX_Q(N,s)$) is the Shimura curve associated to $V'_Q(N,s)$ (resp. $V_Q(N,s)$).\
There is a natural surjection $\sigma_Q:\TT_Q^{'\widehat\psi}(N,s)\to \TT_Q^{\widehat\psi}(N,s)$. Since the diamond operator $\langle n\rangle$ depends only on the image of $n$ in $\Delta_Q$, $\TT_Q^{'\widehat\psi}(N,s)$ is naturally an $\OO[\Delta_Q]$-algebra.
Let $\widetilde\alpha_{i,q}$ for $i=1,2$ be the two  roots in $\OO$ of the polinomial $X^2-a_q(f)X+q$ reducing to $\alpha_{i,q}$ for $i=1,2$. 
There is a unique eigenform $\widetilde f_Q\in S_2(\Gamma_0(MQs^2\ell^2),\widehat\psi)$ such that $\rho_{\widetilde f_Q}=\rho_f$, $a_s(\widetilde f_Q)=0$, $a_q(\widetilde f_Q)=\widetilde\alpha_{2,q}$ for $q|Q$, where $\widetilde\alpha_{2,q}$ is the lift of $\alpha_{2,q}$.\\
By the Jacquet-Langlands correspondence, the form $\widetilde f_Q$ determines a character $\theta_Q:\TT_Q^{\widehat\psi}(N,s)\to k$ sending $T_p$ to $a_p(\widetilde f_Q)\ \modulo\ \lambda$ and the diamond operators to 1. We define $\widetilde\mm_Q=\ker \theta_Q,$ $\mm_Q=\sigma^{-1}_Q(\widetilde\mm_Q)$, and $\MM_Q=H^1(\Phi'_Q,\OO(\ts))_{\mm_Q}.$
Then the map $\sigma_Q$ induce a surjective homomorphism $\TT_Q^{'\widehat\psi}(N,s)_{\mm_Q}\to \TT_Q^{\widehat\psi}(N,s)_{\widetilde\mm_Q}$ whose kernel contains $I_Q(\TT_Q^{'\widehat\psi}(N,s))_{\mm_Q}$. \\
If $\mathcal Q$ is a family of finite sets $Q$ of primes satisying conditions $(A)$ e $(B)$, conditions (TWS1) and (TWS2) holds, as proved in \cite{Lea}, proposition 3.2; by the same methods as in \S 6 of \cite{DDT1} and in \S 4, \S 5 of \cite{Des}, it is easy to prove that our system $(\mR_Q,\MM_Q)_{Q\in\mathcal Q}$ realize simultaneously conditions (TWS3), (TWS4), (TWS5).  \\
We put $\delta_q=\left(
\begin{array}
[c]{cc}%
q & 0\\
0 & 1
\end{array}
\right).$ 
Let $\eta_q$ be the id\`{e}le in $B_\A^\s$ defined by $\eta_{q,v}=1$ if $v\not|q$ and $\eta_{q,q}=i^{-1}_q\left(
\begin{array}
[c]{cc}%
q & 0\\
0 & 1
\end{array}
\right).$ By strong approximation, write $\eta_q=\delta_qg_\infty u$ with $\delta_q\in B^\s$, $g_\infty\in GL_2^+(\RR)$, $u\in V_{Q'}(N,s)$. We define a map 
\begin{eqnarray}
H^1(\Phi_Q,\OO(\ts))&\to& H^1(\Phi_{Q'},\OO(\ts))\\
x&\to& x|_{\eta_q}
\end{eqnarray}
as follows: let $\xi$ be a cocycle representing the cohomology class $x$ in $H^1(\Phi_{Q},\OO(\ts))$; then $x|_{\eta_q}
$ is represented by the cocycle $$\xi|_{\eta_q}(\gamma)=\widehat\psi(\delta_q)\cdot\xi(\delta_q\gamma\delta_q^{-1}).$$ 
We observe that if $\mathcal Q$ is a family of finite sets $Q$ of primes satisying conditions $(A)$ e $(B)$, then condition (TWS6) hold for the system $(\mR_Q,\MM_Q)_{Q\in\mathcal Q}$. The proof is essentially the same as in \cite{Lea}, using the following lemma:
\begin{lemma} 
$$T_p(x|_{\eta_q})=(T_p(x)|_{\eta_q})\ \ \ {\rm if}\ \ p\not|MQ'\ell,$$
\begin{equation}\label{r}
T_q(x|_{\eta_q})=q\widehat\psi(q)res_{\Phi_Q/\Phi_{Q'}}x,
\end{equation}
$$T_q(res_{\Phi_Q/\Phi_{Q'}}x)=res_{\Phi_Q/\Phi_{Q'}}(T_q(x))-x|_{\eta_q}.$$
\end{lemma}
\proof We prove (\ref{r}). We put  $\widetilde\delta_q=\left(
\begin{array}
[c]{cc}%
1 & 0\\
0 & q
\end{array}
\right),$ and we decompose the double coset $\Phi_{Q'}\widetilde\delta_q\Phi_{Q'}=\coprod_{i=1}^q\Phi_{Q'}\widetilde\delta_qh_i$ with $h_i\in\Phi_{Q'}$; if $\gamma\in\Phi_{Q'}$, then:
\begin{eqnarray}
T_q(\xi|_{\eta_q})(\gamma)&=&\sum_{i=1}^q\widehat\psi(h_i\widetilde\delta_q)\xi|_{\eta_q}(\widetilde\delta_qh_i\gamma h_{j(i)}^{-1}\widetilde\delta^{-1}_q)\nonumber\\
&=&\sum_{i=1}^q\widehat\psi(h_i)\widehat\psi(\widetilde\delta_q)\widehat\psi(\delta_q)\xi(\delta_q\widetilde\delta_qh_i\gamma h_{j(i)}^{-1}\widetilde\delta^{-1}_q\delta^{-1}_q)\nonumber\\
&=&\widehat\psi(q)\sum_{i=1}^q\widehat\psi(h_i)\xi(h_i\gamma h_{j(i)}^{-1})
\end{eqnarray}
where the latter sum holds since $\delta_q\widetilde\delta_q=\left(
\begin{array}
[c]{cc}%
q & 0\\
0 & q
\end{array}
\right)$. From the cocycle relations we have $T_q(x|_{\eta_q})=q\widehat\psi(q)res_{\Phi_Q/\Phi_{Q'}}x.$\qed

\noindent The conditions defining the functor $\mathcal F_Q$, characterize a global Galois deformation problem with fixed determinant (\cite{M}, \S 26). We let $ad^0\orho$ denote the subrepresentation of the adjoint representation of $\orho$ over the space of the trace-0-endomorphisms and we let $ad^0\orho(1)=\Hom(ad^0\orho,\mu_p)\simeq Symm^2(\orho)$, with the action of $G_p$ given by $g\varphi)(v)=g\varphi(g^{-1}v).$ Local deformation conditions a$_Q$), b), c), d)  allow one to define for each place $v$ of $\QQ$, a subgroup $L_v$ of $H^1(G_v, ad^0\orho)$, the tangent space of the deformation functor (see \cite{M}). We will describe the computation of the local terms of the dimension formula coming from the Poitou-Tate sequence:
 
\begin{itemize}
\item $\dim_k(H^0(G_\QQ,ad^0\orho)=\dim_k(H^0(G_\QQ,ad^0\orho(1))=0$, by the same argument as in \cite{Des} pp.441.
\item $\dim_kL_\ell=1$. In fact let $\RR^D_{\OO,\ell}$ be the local universal deformation ring associated to a local deformation problem of being weakly of type $\tau$ \cite{CDT}. Since, in dimension 2, potentially Barsotti-Tate is equivalent to potentially crystalline (hence potentially semi stable) of Hodge-Tate weight $(0,1)$ (see \cite{FM}, theorem C2), this allow us to apply Savitt's result (\cite{savitt}, theorem 6.22). Since under our hypothesis $\RR^D_{\OO,\ell}\not=0$, we deduce that there is an isomorphism $
\OO[[X]]\simeq\RR^D_{\OO,\ell}$.
\item  $\dim_kH^0(G_\ell,ad^0\orho)=0$, because of hypothesis (\ref{end})
\item  $\dim_kH^1(G_p/I_p,(ad^0\orho)^{I_p})-\dim_kH^0(G_p,ad^0\orho)=0$ for $p|N\Delta_1$\\
If we let $W=ad^0\orho$, this follows from the exact exact sequence 
\begin{displaymath}
0\to H^0(G_p,W)\to H^0(I_p,W)\stackrel{\scriptstyle\Frob_p-1}{\to}H^0(I_p,W)\to H^1(G_p/I_p,W^{I_p})\to 0.
\end{displaymath} 
\item $\dim_kL_p=1$ for $p|\Delta_2$, infact the following lemma holds:
\begin{lemma}\label{versal}
The versal defomation ring of the local defomation problem of satisfying the sp-condition is $\OO[[X,Y]]/(X,XY)=\OO[[Y]].$
\end{lemma}
\item $\dim_kH^0(G_p,ad^0\orho)=1$, since the eigenvalues of $\orho(\Frob_p)$ are distinct.
\item $\dim_kH^1(G_q,ad^0\orho)=2$ if $q|Q$, in fact $H^1(G_q/I_q,W)=W/(\Frob_q-1)W$ has dimension 1, because in the base $\left\{\left(
\begin{array}
[c]{cc}%
1 & 0\\
0 & -1
\end{array}
\right), \left(
\begin{array}
[c]{cc}%
0 & 0\\
0 & 1
\end{array}
\right), \left(
\begin{array}
[c]{cc}%
0 & 0\\
0 & 1
\end{array}
\right)\right\}$ of $W$, $\Frob_q$ is the matrix 
$$\left(
\begin{array}
[c]{ccc}%
1 & 0 & 0\\
0 & \alpha_{1,q}\alpha_{2,q}^{-1} & 0\\
0 & 0 & \alpha_{1,q}^{-1}\alpha_{2,q}
\end{array}
\right)$$
 and $\alpha_{1,q}\not=\alpha_{2,q}$ by hypothesis. We observe that:
\begin{eqnarray}
H^1(I_q,W)^{G_q/I_q}&=&\{\alpha\in \Hom(\ZZ_q^\s,W)\ |\ (\Frob_q-1)\alpha=0\}\nonumber\\
&\simeq& W[\Frob_q-1]\nonumber
\end{eqnarray}
is again one-dimensional, and $H^2(G_q/I_q,W)=0$ since  $G_q/I_q\simeq\widehat{\ZZ}.$ The desidered result follows from the inflation-restriction exact sequence.
\item $\dim_kH^0(G_q,ad^0\orho)=1$ if $q|Q$, since the eigenvalues of $\Frob_q$ are $1, \alpha_{1,q}^2, \alpha_{2,q}^2$ and $\alpha_{1,q}^2, \alpha_{2,q}^2\not=1$ for hypothesis.
\item $\dim_kH^1(G_\infty,ad^0\orho)=0$, since $|G_\infty|=2\not=\ell$.
\item $\dim_kH^0(G_\infty,ad^0\orho)=1$, since the eigenvalues of complex conjugation on $ad^0\orho$ are $\{1,-1,-1\}.$
\end{itemize}

\proof[lemma \ref{versal}]

\noindent We first observe that it is possible to  characterize the deformations of $\orho_p$ in the unramified case, if $p^2\not\equiv 1\ \modulo\ \ell$ and, as in Lemma 2.1 in \cite{Lea}, it it is easy to prove that the versal deformation ring $\mR'_p$ is generated by two elements $X,Y$ such that $XY=0$. \\
It is immediate to see that the sp-condition is equivalent to require that every homomorphism $\varphi:\mR'_p\to A$ associated to the deformation $\rho$ of $\orho_{G_p}$ over a $\OO$-algebra $A\in\mathcal C_\OO$, has $\varphi(X)=0$.\qed

\noindent The dimension formula allow us to obtain the following identity: 
\begin{eqnarray}\label{f}
\dim_kSel_Q(ad^0\orho)-\dim_kSel^*_Q(ad^0\orho(1))=|Q|
\end{eqnarray}
and, since the minimal number of topological generators of $\mR_Q$ is equal to $\dim_kSel_Q(ad^0\orho)$,
 we obtain that the $\OO$-algebra $\mR_Q$ can be generated topologically by $|Q|+\dim_kSel^*_Q(ad^0\orho(1))$ elements.
Applying the same arguments as in $\cite{DDT}$ the proof of theorem \ref{goal} follows. 

\section{Quaternionic congruence module}

In this section we will keep the notations as in the previous sections.\\
As remarked in section \ref{shi}, there is an injective map $$JL:S_2(V_0(N),\widehat\psi)\to S_2(\Gamma_0(\Delta')\cap\Gamma_1(N\ell^2))$$ where $V_0(N)=\prod_{p\not|N}R_p^\s\s\prod_{p|N}K_p^0(N)$ is an open compact subgroup of $B_\A^{\s,\infty}$.  We put $V_{\widehat\psi}=JL(S_2(V_0(N),\widehat\psi))$ the subset of $S_2(\Gamma_0(\Delta')\cap\Gamma_1(N\ell^2))$ generated by thouse new eigenforms with nebentypus $\psi$ which are supercuspidal of type $\tau$ at $\ell$ and special at primes dividing $\Delta'$. Let $K\subset\overline\QQ_\ell$ be a finite extension containing $\QQ_{\ell^2}$; we consider $f\in S_2(\Gamma_0(N\Delta'\ell^2),\psi)$ a newform, supercuspidal of type $\tau$ at $\ell$ and special at primes dividing $\Delta'$, and let $X$ be the  subspace of $V_{\widehat\psi}(K)$ spanned by $f$. We remark, that there is an isomorphism between the $K$-algebra $\TT_0^{\widehat\psi}(K)=\TT_0^{\widehat\psi}(N)\otimes_\OO K$ generated over $K$ by the operators $T_p$, $p\not=\ell$ acting on $H^1(\XX_1(N),K)^{\widehat\psi}$ and the Hecke algebra generated by all the Hecke operators acting on $V_{\widehat\psi}(K)$. Thus $V_{\widehat\psi}=X\oplus Y$ where $Y$ is the orthogonal complement of $X$ with respect to the Petersson product over $V_{\widehat\psi}$ and there is an isomorphism $$\TT_0^{\widehat\psi}(K)\simeq\TT_{0}^{\widehat\psi}(K)_X\oplus\TT_{0}^{\widehat\psi}(K)_Y$$ where $\TT_{0}^{\widehat\psi}(K)_X=\TT_0^{\widehat\psi}(N)_X\otimes K$ and $\TT_{0}^{\widehat\psi}(K)_Y=\TT_0^{\widehat\psi}(N)_Y\otimes K $ are the $K$ algebras generated by the Hecke operators acting on $X$ and $Y$ respectively.\\
As in classical case, it is possible to define the  quaternionic congruence module for $f$ and, by the Jacquet-Langlands correspondence, it is easy to prove that $\widetilde M^{\rm quat}=\OO/(\lambda^n)$ where $n$ is the smallest integer such that $\lambda^ne_f\in\TT_0^{\widehat\psi}(N)$ where $e_f$ is the projector onto the coordinate corresponding to $f$. There are the isomorphisms  
$$\widetilde M^{\rm quat}\simeq\frac{\TT_0^{\widehat\psi}(N)_X\oplus\TT_0^{\widehat\psi}(N)_Y}{\TT_0^{\widehat\psi}(N)}\simeq\frac{e_f\TT_0^{\widehat\psi}(N)}{e_f\TT_0^{\widehat\psi}(N)\cap\TT_0^{\widehat\psi}(N)}$$
where the first one is an isomorphism of $\OO$-modules and the second one is obtained considernig the projection map of $\TT_0^{\widehat\psi}(N)_X\oplus\TT_0^{\widehat\psi}(N)_Y$ onto the first component.\\
Now, let $\mm$ be the maximal ideal of $\TT_0^{\widehat\psi}(N)$ definend in section \ref{tw}, since $e_f\TT_0^{\widehat\psi}(N)_\mm=e_f\TT_0^{\widehat\psi}(N)$ then, by the results in the previous sections  $$\widetilde M^{\rm quat}=\OO/(\lambda^n)=\frac{e_f\TT^\psi}{e_f\TT^\psi\cap\TT^\psi}=(\widetilde L^{\rm quat})^2$$  where $\widetilde L^{\rm quat}$ is the  quaternionic cohomological congruence module for $f$ $$\widetilde L^{\rm quat}=\frac{e_fH^1(\XX_1(N),\OO)^{\widehat\psi}}{e_fH^1(\XX_1(N),\OO)^{\widehat\psi}\cap H^1(\XX_1(N),\OO)^{\widehat\psi}}.$$

\section{A generalization of the Conrad, Diamond and Taylor's result using Savitt's theorem}\label{gen}

In \cite{CDT}, Conrad, Diamond and Taylor assume that the type $\tau$ is strongly acceptable for $\orho|_{G_\ell}$ and they consider the global Galois deformation problem of being of {\bf type $(S,\tau)$ }, where $S$ is a set of rational primes which does not contain $\ell$. Savitt's theorem allows to suppress the assumption of acceptability in the definition of strong acceptability and their result, theorem 5.4.2., still follows. 
They first suppose that $S=\emptyset$ and they prove their result using the improvement on the method of Taylor and Wiles \cite{taywi} found by Diamond \cite{D} and Fujiwara \cite{Fu}, then they prove their result for an arbitrary $S$ by induction on $S$, using Ihara's Lemma.\\
In particular, if $S$ is a set of rational prime not dividing $N\Delta'\ell$, we consider a newform $f$ of weight 2, level dividing $SM\ell^2$, with nebentypus $\psi$ (not trivial), supercuspidal of type $\tau=\chi\oplus\chi^\sigma$ at $\ell$  and such that $\orho_f=\orho$ satisfies the conditions (\ref{con1}),(\ref{cond2}),(\ref{end}) and (\ref{c3}) of section \ref{de}. As a general hypothesis, we assume that $f$ occurs with type $\tau$ and minimal level, that is $\orho_f$ is ramified at every prime in $\Delta'$.\\
We consider   deformations of type $(S,\tau)$ of $\orho_f$ unramified outside the level of $f$ and  such that $\det(\rho)=\epsilon\psi$; we will call this deformation problem of type $(S,\tau,\psi)$. Then Savitt'theorem assure that the tangent space of the deformation functor at $\ell$ is still one-dimensional and so it is possible to go on with the same construction as in \cite{CDT}.
Let $\mathcal R_S^{\rm mod,\psi}$ be classical type $(S,\tau,\psi)$ universal deformation ring which pa\-ra\-me\-trizes representations of type $(S,\tau,\psi)$ with residual representation $\orho$ and let $\TT_S^{\rm mod,\psi}$ be the classical Hecke algebra acting on the space of the modular forms of type $(S,\tau,\psi)$. If we denote by $\mathcal M_S^{\rm mod}$ the cohomological module defined in \S 5.3 of \cite{CDT}, which is essentially the \lq\lq$\tau$-part\rq\rq  of the first cohomology group of a modular curve of level depending on $S$, let $\mathcal M^{\rm mod,\psi}_S$ be the $\psi$-part of $\mathcal M_S^{\rm mod}$. Then the following proposition follows:
\begin{proposition}\label{gcdt}
  The map $$\Phi_S^{\rm mod,\psi}:\mathcal R_S^{\rm mod,\psi}\to\TT_S^{\rm mod,\psi}$$ is a complete intersection isomorphism and $\mathcal M^{\rm mod,\psi}_S$ is a free $\TT_S^{\rm mod,\psi}$-module of rank 2. 
\end{proposition}
\noindent In particular we observe that:
\begin{lemma}
There is an isomorphism of $\TT^\psi$-modules between $H^1(\XX_1(N),\OO)^{\widehat\psi}_\mm$ and $\mathcal M^{{\rm mod},\psi}_\emptyset$ where $\XX_1(N)$ is the Shimura curve associated to $V_1(N)=\prod_{p\not|N\ell}R_p^\s\prod_{p|N}K_p^1(N)\s(1+u_\ell R_\ell)$.
\end{lemma}
\proof
We observe that if $f$ occurs with type $\tau$ and minimal level $\mR\simeq\mathcal R^{{\rm mod},\psi}_\emptyset$. By theorem \ref{goal} and proposition \ref{gcdt}, there is an isomorphism between the Hecke algebras $\TT^\psi\simeq\TT^{{\rm mod},\psi}_\emptyset$ thus $\mM\simeq \mathcal M^{{\rm mod},\psi}_\emptyset$ as $\TT^\psi$-modules.\qed

\noindent We will describe some consequences of this result.

\section{Congruence ideals}\label{con}

In this section we generalize the results about congruence ideals of Terracini \cite{Lea}, considering modular forms with nontrivial nebentypus.\\
Let $\Delta_1$ be a set of primes, disjoint from $\ell$. By an abuse of notation, we shall sometimes denote by $\Delta_1$ also the product of the prime in this set.\\
As before, let $f$ be a newform in $S_2(\Gamma_0(N\Delta_1\ell^2),\psi))$,  supercuspidal of type $\tau$ at $\ell$ and as a general hypothesis we assume that the residual representation $\orho$, associated to $f$ occurs with type $\tau$ and minimal level. \\
We observe that if  $\orho_\ell$  has the form as at pp.525 of \cite{CDT}, then $f$ satisfies the above hypothesis.\\ 
We assume that the character $\psi$ satisfies the condition as in the section \ref{de}, that $\orho$ is absolutely irreducible and that $\orho_\ell$ has trivial centralizer.\\
Let $\Delta_2$ be a finite set of primes $p$, not dividing $\Delta_1\ell$ such that $p^2\not\equiv 1\ \modulo\ \ell$ and $\tr(\orho(\Frob_p))^2\equiv\psi(p)(p+1)^2\ \modulo\ \ell$. We let $\mathcal B_{\Delta_2}$ denote the set of newforms $h$ of weight 2, character $\psi$ and level dividing $N\Delta_1\Delta_2\ell$ which are special at $\Delta_1$, supercuspidal of type $\chi$ at $\ell$ and such that $\orho_h=\orho$. We choose an $\ell$-adic ring $\OO$ with residue field $k$, sufficiently large, so that every representation $\rho_h$ for $h\in\mathcal B_{\Delta_2}$ is defined over $\OO$ and $Im(\psi)\subseteq\OO$. For every pair of disjoint subset $S_1, S_2$ of $\Delta_2$ we denote by $\mathcal R^\psi_{S_1,S_2}$ the universal solution over $\OO$ for the deformation problem of $\orho$ consisting of the deformations $\rho$ satisfying:
\begin{itemize}
\item[a)] $\rho$ is unramified outside $N\Delta_1 S_1S_2\ell$;
\item[b)] if $p|\Delta_1N$ then $\rho(I_p)=\orho(I_p)$;
\item[c)] if $p|S_2$ then $\rho_p$ satisfies the sp-condition;
\item[d)] $\rho_\ell$ is weakly of type $\tau$;
\item[e)] $\det(\rho)=\epsilon\psi$ where $\epsilon:G_\QQ\to\ZZ^\s_\ell$ is the cyclotomic character. 
\end{itemize}
Let $\mathcal B_{S_1,S_2}$ be the set of newforms in $\mathcal B_{\Delta_2}$ of level dividing $N\Delta_1S_1S_2\ell$ which are special at $S_2$ and let $\TT^\psi_{S_1,S_2}$ be the sub-$\OO$-algebra of $\prod_{h\in\mathcal B_{S_1,S_2}}\OO$ generated by the elements $\widetilde T_p=(a(h))_{h\in\mathcal B_{S_1,S_2}}$ for $p$ not in $\Delta_1\cup S_1\cup S_2\cup\{\ell\}$. Since $\mathcal R^\psi_{S_1,S_2}$ is generated by traces, we know that there exist a surjective homomorphism of $\OO$-algebras $\mathcal R^\psi_{S_1,S_2}\to\TT^\psi_{S_1,S_2}$. Moreover by the results obtained in section \ref{gen}, we have that $\mathcal R^\psi_{S_1,\emptyset}\to\TT^\psi_{S_1,\emptyset}$ is an isomorphism of complete intersections, for any subset $S_1$ of $\Delta_2$.\\
If $\Delta_1\not=1$ then each $\TT^\psi_{\emptyset,S_2}$ acts on a local component of the cohomology of a suitable Shimura curve, obtained by taking an indefinite quaternion algebra of discriminant $S_2\ell$ or $S_2\ell p$ for a prime $p$ in $\Delta_1$. Therefore, theorem \ref{goal} gives the following:
\begin{corollary}
Suppose that $\Delta_1\not=1$ and that $\mathcal B_{\emptyset,S_2}\not=\emptyset;$ then the map $$\mathcal R^\psi_{\emptyset,S_2}\to\TT^\psi_{\emptyset,S_2}$$ is an isomorphism of complete intersections.
\end{corollary}
If $p\in S_2$ there is a commutative diagram:
\begin{eqnarray}
&\mathcal R^\psi_{S_1p,S_2/p}& \to \mathcal R^\psi_{S_1,S_2}\nonumber\\
&\downarrow &\ \ \ \ \ \downarrow\nonumber\\
&\TT^\psi_{S_1p,S_2/p}&\to \TT^\psi_{S_1,S_2}\nonumber\\
\end{eqnarray}
where all the arrows are surjections.\\
For every $p|\Delta_2$ the deformation over $\mathcal R^\psi_{\Delta_2,\emptyset}$ restricted to $G_p$ gives maps $$\mathcal R^{\psi}_p=\OO[[X,Y]]/(XY)\to\mathcal R^\psi_{\Delta_2,\emptyset}.$$ The image $x_p$ of $X$ and the ideal $(y_p)$ generated by the image $y_p$ of $Y$ in $\mathcal R^\psi_{\Delta_2,\emptyset}$ do not depend on the choice of the map. By an abuse of notation, we shall call $x_p,y_p$ also the image of $x_p,y_p$ in every quotient of $\mathcal R^\psi_{\Delta_2,\emptyset}$. If $h$ is a form in $\mathcal B_{\Delta_2,\emptyset}$, we denote by $x_p(h),y_p(h)\in\OO$ the images of $x_p,y_p$ by the map $\mathcal R^\psi_{\Delta_2,\emptyset}\to\OO$ corresponding to $\rho_h$.\\
\begin{lemma}
If $h\in\mathcal B_{\Delta_2}$ and $p|\Delta_2$, then:
\begin{itemize}
\item[a)] $x_p(h)=0$ if and only if $h$ is special at $p$;
\item[b)] if $h$ is unramified at $p$ then $(x_p(h))=(a_p(h)^2-\psi(p)(p+1)^2)$;
\item[c)] $y_p(h)=0$ if and only if $h$ is unramified at $p$;
\item[d)] if $h$ is special at $p$, the oreder at $(\lambda)$ of $y_p(h)$ is the greatest positive integer $n$ such that $\rho_h(I_p)\equiv\sqrt{\psi(p)}\otimes 1\ \modulo\ \lambda^n.$
\end{itemize}
\end{lemma}

\begin{proof}
It is an immediate consequence of the definition of the sp-condition (proof of lemma \ref{versal}). Statement b), follows from the fact that $$a_p(h)=\tr(\rho_h(\Frob_p))$$ and $$\rho_h(\Frob_h)=\left(
\begin{array}
[c]{cc}%
\pm p\sqrt{\psi(p)}+x_p(h) & 0\\
0 & p\psi(p)/(\pm p\sqrt{\psi(p)}+x_p(h))
\end{array}
\right).$$ \end{proof}

\noindent In particular $(y_p)$ is the kernel of the map $\mathcal R^\psi_{S_1,S_2}\to\mathcal R^\psi_{S_1/p,S_2}$.

\noindent If $h\in\mathcal B_{S_1,S_2}$ let $\theta_{h,S_1,S_2}:\TT^\psi_{S_1,S_2}\to\OO$ be the character corresponding to $h$.

\noindent We consider the congruence ideal of $h$ relatively to $\mathcal B_{S_1,S_2}$: $$\eta_{h,S_1,S_2}=\theta_{h,S_1,S_2}(Ann_{\TT^\psi_{S_1,S_2}}(\ker\ \theta_{h,S_1,S_2})).$$

\noindent It is know that $\eta_{h,S_1,S_2}$ controls congruences between $h$ and linear combinations of forms different from $h$ in $\mathcal B_{S_1,S_2}$.

\begin{theorem}\label{cong}
Suppose $\Delta_1\not=1$ and $\Delta_2$ as above. Then
\begin{itemize}
\item[a)] $\mathcal B_{\emptyset,\Delta_2}\not=0$;
\item[b)] for every subset $S\subseteq\Delta_2$, the map $\mathcal R^\psi_{S,\Delta_2/S}\to\TT^\psi_{S,\Delta_2/S}$ is an isomorphism of complete intersection;
\item[c)] for every $h\in\mathcal B_{\emptyset,\Delta_2}$, $\eta_{h,S,\Delta_2/S}=(\prod_{p|S}y_p(h))\eta_{h,\emptyset,\Delta_2}.$
\end{itemize}
\end{theorem}

\noindent The proof of this theorem is essentilly the same as in \cite{Lea}.\\
If we combine point c) of theorem \ref{cong} to the results in Section 5.5 of \cite{CDT}, we obtain:
\begin{corollary} 
If $h\in\mathcal B_{S_1,S_2},$ then
\begin{displaymath}
  \eta_{h,\Delta_2,\emptyset}=\prod_{p|\frac{\Delta_2}{S_1S_2}}x_p(h)\prod_{p|S_2}y_p(h)\eta_{h,S_1,S_2}.
  \end{displaymath}
\end{corollary}
In particular, from this corollary we deduce the following theorem:
\begin{theorem}
Let $f=\sum a_nq^n$ be a normalized newform in $S_2(\Gamma_0(M\ell^2),\psi)$ supercuspidal of type $\tau=\chi\oplus\chi^\sigma$ at $\ell$, with minimal level, special at primes in a finite set $\Delta'$, there exist $g\in S_2(\Gamma_0(qM\ell^2),\psi)$ supercuspidal of type $\tau$ at $\ell$, special at every prime $p|\Delta'$ such that $f\equiv g\ \modulo\ \lambda$ if and only if $$a_q^2\equiv\psi(q)(1+q)^2\ \modulo\ \lambda$$ where $q$ is a prime such that $(q,M\ell^2)=1,$ $q\not\equiv-1\ \modulo\ \ell$. 
\end{theorem}

\noindent The problem of remuve the hypothesis of minimal level, is still open and could be solved by proving conjecture \ref{noscon}

\section{Problem: extension of results to the non minimal case}

Let $\ell\geq 2$ be a prime number. Let $\Delta'$ be a product of an odd number of primes, different from $\ell$. We put $\Delta=\Delta'\ell$. Let $B$ be the indefinite quaternion algebra over $\QQ$ of discriminant $\Delta$. Let $R$ be a maximal order in $B$.\\
 Let $N$ be a positive rational number, we observe that in our deformation problem, in section \ref{de}, we have assumed that the representation $\orho$ associated to $f\in S_2(\Gamma_0(N\Delta\ell),\psi)$, occurs with type $\tau$ and minimal level at $N$ ( not necessarily at $\Delta'$). Let now $S$ be a finite set of rational primes which does not divide $M\ell$, where $M=N\Delta'$, we fix $f\in S(\Gamma_0(N\Delta\ell S),\psi)$ a modular newform of weight 2, level $N\Delta\ell S$, supercuspidal of type $\tau$ at $\ell$, special at primes dividing $\Delta'$  and with Nebentypus $\psi$. Let $\rho$ be the Galois representation associated to $f$ and let $\orho$ be its reduction modulo $\ell$; we suppose that conditions (\ref{con1}), (\ref{cond2}), (\ref{rara}), (\ref{end}) and (\ref{c3}) hold.
\noindent If we denote by $\Delta_1$ the product of primes $p|\Delta'$ such that $\orho(I_p)\not=1$ and by $\Delta_2$ the product of primes $p|\Delta'$ such that $\orho(I_p)=1$. As usual we assume that $p^2\not\equiv 1\ \modulo\ \ell$ if $p|\Delta_2$. We say that the representation $\rho$ is of {\bf type $(S,{\rm sp},\tau,\psi)$} if conditions b), c), d), e) of definition \ref{def} hold and
\begin{itemize}
\item[a)] $\rho$ is unramified outside $M\ell S$.
\end{itemize} 

\noindent This is a deformation condition. Let $\mathcal R^\psi_S$ be the universal deformation ring which parametrizes representations of type $(S,{\rm sp},\tau,\psi)$ with residual representation $\orho$ and let $\TT^\psi_S$ be the Hecke algebra acting on the space of modular forms of type $(S,{\rm sp},\tau,\psi)$.\\ 
Since the dimension of the Selmer groups do not satisfy the control conditions,  it is not possible to construct a Taylor-Wiles system by considering a deformation problem of type $(S,{\rm sp},\tau,\psi)$ with $S\not=\emptyset$, as in section \ref{tw}; a still open problem is  to prove theorem \ref{goal} for $\mathcal R^\psi_S$, $\TT^\psi_S$ and $\mM_S$. If $S=0$ then we retrouve our theorem \ref{goal}. A possible approach to this problem is to use, as in the classical case, the results of De Smit, Rubin,  Schoof \cite{DRS} and Diamond \cite{Dia}, by induction on the cardinality of $S$. If we make the inductive hypothesis assuming the result true for $S$, we have to verify that the result holds for $S'=S\cup\{q\}$ where $q$ is a prime number not dividing $M\ell S$. Following the litterature, to prove the inductive step the principal ingredients are:
\begin{itemize}
\item a duality result about $\mathcal M^\psi_S$ (to appeare), showing the existence  of a perfect pairing $\mathcal M^\psi_S\s \mathcal M^\psi_S\to\OO$ which induces an isomorphism $\mathcal M^\psi_S\to\Hom_\OO(\mathcal M^\psi_S,\OO)$ of $\TT^\psi_S$-modules;
\item conjecture \ref{noscon}, saying that the natural $\TT^\psi_{S'}$-injection $\left(\mathcal M^\psi_S\right)^2\hookrightarrow \mathcal M^\psi_{S'}$ is injective when tensorized with $k$.
\end{itemize} 
For semplicity, we shall assume that $q^2\not\equiv 1\ \modulo\ \ell$ for every $q\in S$.

\noindent Then, as observed in the previous section, there is an isomorphism: $$\frac{\mathcal R^\psi_{S'}}{(y_q)}\ \widetilde\to\ \mathcal R^\psi_{S}$$ that induces an isomorphism on the Hecke algebras: $$\frac{\TT^\psi_{S'}}{(y_q)_{\TT^\psi_{S'}}}\ \widetilde\to\ \TT^\psi_{S}$$
where $(y_q)_{\TT^\psi_{S'}}$ is the image of $(y_q)$ in $\TT^\psi_{S'}$; but, for a general $S$, we don't have any information about $\mathcal M^\psi_S=H^1(\XX_1(NS),\OO)_{\mm_S}^{\widehat\psi}$ as a $\TT_S^{\widehat\psi}$-module, where $\mm_S$ is the inverse image of $\mm$ with respect to the natural map $\TT_S^{\widehat\psi}\to \TT_\emptyset^{\widehat\psi}$.

\end{document}